\documentclass{amsart}
\usepackage{amsfonts}
\usepackage{amsmath}
\usepackage{amssymb}
\usepackage{graphicx}
\usepackage{float}

\newtheorem{definition}{\bf Definition}
\newtheorem{example}{\bf Example}
\newtheorem{lemma}{\bf Lemma}
\newtheorem{theorem}{\bf Theorem}
\newtheorem{remark}{\bf Remark}
\newtheorem{corollary}{\bf Corollary}

\begin{document}

\title[Levi graphs of point-line configurations]{On the Levi graph of point-line configurations}
\author{Jessica Hauschild}
\address{Kansas Wesleyan University. Salina, KS.}
\email{jessica.hauschild@kwu.edu}
\author{Jazmin Ortiz}
\address{Harvey Mudd College. Claremont, CA.}
\email{jortiz@g.hmc.edu}
\author{Oscar Vega}
\address{California State University, Fresno. Fresno, CA.}
\email{ovega@csufresno.edu}

\subjclass[2010]{Primary 05B30; Secondary 51E05, 51E30}
\keywords{Levi graph, maximal independent sets, configurations}
\thanks{We gratefully acknowledge the NSF for their financial support (Grant \#DMS-1156273), and the REU program at California State University, Fresno.}

\begin{abstract}
We prove that the well-covered dimension of the Levi graph of a point-line configuration $(v_r,b_k)$ is equal to $0$, whenever  $r>2$.
\end{abstract}

\maketitle

\section{Introduction}

The concept of the well-covered space of a graph was first introduced in 1998 and 1999 by Caro, Ellingham, Ramey, and Yuster (see \cite{CER} and \cite{CY}) as an effort to  generalize the study of well-covered graphs. Brown and Nowakowski \cite{BN},  in 2005, continued the study of this object and, among other things,  provided several examples of graphs featuring odd behaviors regarding their well-covered space. One of these special situations occurs when the well-covered space of the graph is trivial, i.e. when the graph is \emph{anti-well-covered}. In this work, we prove that almost all Levi graphs of configurations $(v_r, b_k)$ are anti-well-covered.

We start our exposition by providing the following definitions and previously-known results. Any introductory concepts failed to be defined here may be found in the books by Bondy \& Murty \cite{BM} and Gr\"unbaum \cite{G}.

We consider only simple and undirected graphs. A graph will be denoted as $G=(V(G), E(G))$, as is customary, where $V(G)$ is the set of vertices of the graph and $E(G)$ is the set of edges of the graph. Two vertices of a graph are said to be \textit{adjacent} if they are connected by an edge. An \textit{independent} set of vertices is one in which no two vertices in the set are adjacent. If an independent set, $M$, of a graph $G$, is not a proper subset of any other independent set of $G$, then $M$ is a \textit{maximal independent set} of $G$.

\begin{definition}
Let $G$ be a graph and $\mathbf F$ a field.
\begin{enumerate}
\item A function $f:V(G) \rightarrow \mathbf F$ is said to be a \textit{weighting} of $G$. If the sum of all weights is constant for all maximal independents sets of $G$, then the weighting is a \textit{well-covered weighting} of $G$. 
\item The $\mathbf F$-vector space consisting of all well-covered weightings of $G$ is called the well-covered space of $G$ (relative to $\mathbf F$). 
\item The dimension of this vector space is called the \textit{well-covered dimension} of $G$, denoted $wcdim(G, \mathbf F)$. 
\end{enumerate}
\end{definition}

\begin{remark}
For some graphs, the characteristic of the field $\mathbf F$ makes a difference when calculating the well-covered dimension (see \cite{BKMUV} and \cite{BN}). 
If $char(\mathbf F)$ does not cause a change in the well-covered dimension, then the well-covered dimension is denoted as $wcdim(G)$. 
\end{remark}

In order to calculate the well-covered dimension of a graph, $G$, one would generally need to find all possible maximal independent sets of $G$. However, finding all maximal independent sets is not always an easy task, as this is a known NP-complete problem.

Despite the NP-complete nature of this problem, let us assume that we have found all possible maximal independent sets of $G$. We will denote these maximal independent sets as $M_i$ for $i= 0, 1, \ldots, k-1$. The well-covered weightings of $G$ are determined by solving a system of linear equations that arise from considering all equations of the form $M_0=M_i$ for $i=1, \ldots, k-1$. When subtraction occurs in each of these equations, we get a homogeneous system. We can then form an associated matrix, $A_G$ from this homogeneous system of linear equations. It follows that the nullity of $A_G$ will be the dimension of the well-covered space of $G$. Thus, 
\[
wcdim(G, \mathbf F) = |V(G)|-rank(A_G).
\]

We now move onto another component of our work: configurations. 

\begin{definition}[Gr\"unbaum \cite{G}]\label{defconfiguration}
A configuration is a family of points and lines that satisfy these conditions.
\begin{enumerate}
\item It is a symmetric relation.
\item Incidence is only between a single point and line.
\item Two points are incident with at most one line.
\item Two lines are incident with at most one point.
\end{enumerate}
\end{definition}

Next, there is some notation for configurations that needs to be set, as well as specific parameters that need to be established for the main result of this work.

\begin{definition} \label{defconfigurationpara} 
We define a $(v_r, b_k)$ configuration as a point-line configuration such that
\begin{enumerate}
\item There are exactly $k$ points incident with each line, and $k\geq 2$.
\item There are exactly $r$ lines incident with each point, and $r\geq 2$.
\item There are exactly $v$ points in $(v_r, b_k)$, and $v\geq 4$.
\item There are exactly $b$ lines in $(v_r, b_k)$, and $b\geq 4$.
\end{enumerate}
\end{definition}

When $v=b$ and $r=k$, the configuration will be denoted by $(v_r)$.    

\begin{example}
Several well-known geometric structures fall into the category of $(v_r, b_k)$ configurations. For instance:
\begin{enumerate}
\item A projective plane of order $q$ is a $(q^2+q+1_{(q+1)})$ configuration, where $q$ is the power of a prime. See Figure \ref{fig:PG(2,3)} for a representation of $PG(2,3)=(13_4)$.
\item The Pappus configuration is a $(9_3)$ configuration, and the Desargues configuration is a $(10_3)$ configuration.
\item $PG(n,q)$ is a $\left(\displaystyle \frac{q^{n+1}-1}{q-1}_{(q+1)}, \displaystyle \frac{(q^{n+1}-1)(q^n-1)}{(q^2-1)(q-1)}_{(q^2+q+1)}\right)$ configuration, where $q$ is the power of a prime. 
\item A generalized quadrangle $G(s,t)$ is a $((1+s)(st+1)_{(1+s)}, (1+t)(st+1)_{(1+t)})$ configuration.
\end{enumerate} 
The reader is referred to the book by Batten \cite{B} for more information about these important geometric objects.
\end{example}

Finally, we define the object, Levi graphs, that will connect configurations and graphs. 

\begin{definition}
Given a configuration, $\mathcal{C}$, we define the \textit{Levi graph} of $\mathcal{C}$, denoted $Levi_{\mathcal{C}}$, as the bipartite graph with vertices given by the points and lines in $\mathcal{C}$. The edges of $Levi_{\mathcal{C}}$ connect a point-vertex $P$ with a line-vertex $\ell$ if and only if the point $P$ is incident with the line $\ell$ in $\mathcal{C}$. No other edges exist in $Levi_{\mathcal{C}}$. 
\end{definition}
\begin{figure}[H]
\centering
\raisebox{.6in}{\includegraphics[width=0.55\textwidth]{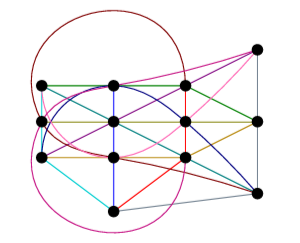}} \hspace{.1in}  \includegraphics[width=0.4\textwidth]{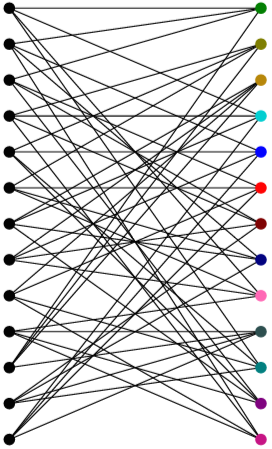}
\caption{$(13_4) = PG(2,3)$ and $Levi_{(13_4)}$}
\label{fig:PG(2,3)}
\end{figure}

Our main result, which will be proven in the following section, combines all these objects as follows:

\begin{theorem} \label{PH(rw)}
If $r\in \mathbb{N}$ and $r>2$, then $wcdim(Levi_{(v_r,b_k)})=0$.
\end{theorem}

\section{The Well-Covered Dimension of $Levi_{(v_r,b_k)}$}

We will prove Theorem \ref{PH(rw)} by first proving a technical lemma, that introduces a family of maximal independent sets that will show to be useful later on.

\begin{lemma}\label{MIS_gq}
A Levi graph of a configuration $(v_r, b_k)$, where $r > 2$, has at least $v + b + 2$ maximal independent sets. 
\end{lemma}

\begin{proof}
Let $P$ be a fixed point in $(v_r, b_k)$. We consider the set, $M_P$, of vertices of $Levi_{(v_r,b_k)}$ given by $P$ and all the lines not incident to $P$.  This is an independent set of $Levi_{(v_r,b_k)}$ because there is no incidence between vertices in the set.  Moreover, note that if we included another point-vertex to $M_P$, then that vertex would be adjacent to one of the line-vertices in $M_P$ (because of condition (2) in Definition \ref{defconfiguration}, and the fact that $r>2$). Also, if another line-vertex where to be added to $M_P$, then this line would have to be incident with $P$. It follows that $M_P$ is a maximal independent set of  $Levi_{(v_r,b_k)}$.
\begin{figure}[H]
\centering
\includegraphics[width=0.35\textwidth, angle=-90]{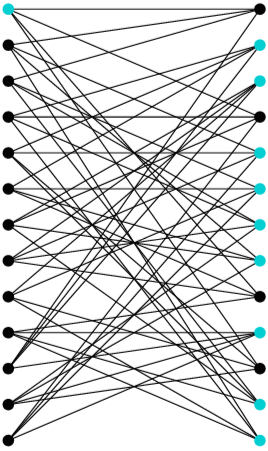}
\caption{A maximal independent set $M_P$ in  $Levi_{(13_4)}$}
\label{fig:Levi1}
\end{figure}
Repeating this construction for all $v$ points in $(v_r,b_k)$, we get $v$ distinct maximal independent sets of $Levi_{(v_r,b_k)}$.

We will now construct another $b$ distinct maximal independent sets of $Levi_{(v_r,b_k)}$. We start by fixing a line $\ell$ in $(v_r,b_k)$ and then any two distinct points $P_1,P_2\in \ell$ (recall that $k \geq 2$). We consider the set, $M_{P_1,P_2}$ of vertices of $Levi_{(v_r,b_k)}$ given by $P_1,P_2$ and all the lines not incident to either of these points.   Note that this forms an independent set since adjacency in $Levi_{(v_r,b_k)}$ only occurs if incidence occurs in $(v_r, b_k)$.  If we try to add in another vertex-point to $M_{P_1,P_2}$, since $r> 2$, then this point will be incident to one of the lines not through $P_1$ or $P_2$ and will therefore be adjacent to the vertex-lines in $M_{P_1,P_2}$. If we try to add another vertex-line to $M_{P_1,P_2}$, then this line will be incident to one or both $P_1$ and $P_2$. Therefore, $M_{P_1,P_2}$ is a maximal  independent of $Levi_{(v_r,b_k)}$. 
\begin{figure}[H]
\centering
\includegraphics[width=0.35\textwidth, angle=-90]{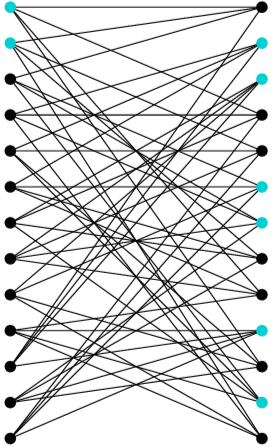}
\caption{A maximal independent set $M_{P_1,P_2}$ in $Levi_{(13_4)}$}
\label{fig:Levi2}
\end{figure}
Repeating this construction for all $b$ lines in $(v_r,b_k)$ (it does not matter what pair of points one picks on any given line), we get $b$ distinct maximal independent sets of $Levi_{(v_r,b_k)}$.

Finally, note that the set of all point-vertices in $Levi_{(v_r,b_k)}$ is a maximal independent set as well as the set of all line-vertices in $Levi_{(v_r,b_k)}$. Hence, we have constructed $v+b+2$ distinct maximal independent sets in $Levi_{(v_r,b_k)}$. 
\end{proof}

Next, we proceed to prove our main result.

\begin{proof}[Proof of Theorem \ref{PH(rw)}]
We denote by $\textbf{F}$ the field of scalars of the well-covered space of $G=Levi_{(v_r,b_k)}$, where $r>2$. Let $A_G$ be the associated matrix of $G$, and note that $A_G$ has $v+b$ columns. In order to prove that $A_G$ has $v+b$ linearly independent rows we will consider the $v+b+2$ maximal independent sets in Lemma  \ref{MIS_gq}.

We create the first $v$ rows of $A_G$ by equating the weight of each of the maximal independent sets $M_P$ to the weight of the maximal independent set consisting of all the lines of $G$. After subtracting we obtain $v$ equations of the form
\begin{equation}\label{eqMISM_P}
f(P) -f(\ell_1) -  f(\ell_2) - \cdots - f(\ell_r) =0 
\end{equation}
where each $\ell_i$ is incident with $P$. It follows that, after organizing the columns of $A_G$ by putting point-vertices first and then line-vertices, the `first' $v$ rows of $A_G$ are
\[
\left[
\begin{array}{cc}
I_{v} & - C \\ [0.2em]
\end{array}
\right]
\]
where $C$ is the incidence matrix of $Levi_{(v_r,b_k)}$. 

In order to obtain the next $b$ rows of $A_G$ we will consider maximal independent sets of the form $M_{P,Q}$. For any given line $\ell$ of $(v_r,b_k)$, we choose (any) two points on it. We will denote these two points as $P_1$ and $P_2$. We then consider the maximal independent set $M_{P_1,P_2}$ and equate its weight to the weight of the maximal independent set $M_{P_1}$.  After subtracting we obtain an equation of the form
\begin{equation}\label{eqMISM_ell}
f(P_2) - f(\ell_1) -  f(\ell_2) - \cdots - f(\ell_r) + f(\ell) =0 
\end{equation}
where each $\ell_i$ is incident with $P_2$.

Note that subtracting Equation \ref{eqMISM_P} (with $P=P_2$)  from Equation \ref{eqMISM_ell} yields $f(\ell) =0$.   Since $\ell$ was arbitrary, we get $f(\ell) =0$  for every line in $(v_r,b_k)$. It follows that since subtracting equations is just a different way to describe row operations in $A_G$, we get that the `first' $v+b$ rows of $A_G$ (after a few row operations) are
\[
\left[
\begin{array}{cc}
I_{v} & - C \\ [0.3em]
\mathbf{0} & I_{b}
\end{array}
\right]
\]

Note that addition and subtraction where the only two (row) operations needed to obtain the matrix above. Hence, the first $v+b$ rows of $A_G$ do not change depending on the characteristic of $\textbf{F}$.

Since the determinant of the matrix above is non-zero,  the rank of $A_G$ is maximal, and thus $wcdim(Levi_{(v_r,b_k)})=0$.
\end{proof}

\section{Possible generalizations}

In this section, we study possible generalizations of Theorem \ref{PH(rw)}. This will be done by providing a few results and by introducing objects for which this theorem could be extended to. We begin by proving that Theorem \ref{PH(rw)} cannot be extended to configurations having exactly two lines being incident with every point. This will be done by an example that considers $(v_2)$ configurations.

We first notice that a $(v_2)$ configuration is a disjoint union of polygons/cycles. This is convenient because disjoint unions of graphs behave well with respect to the well-covered dimension. In fact, Lemma 5 in \cite{BN} says
\[
wcdim(G\cup H) =wcdim(G) + wcdim(H),
\]
where $\cup$ stands for disjoint union.

Since we know that $Levi_{C_n}=C_{2n}$, we get  the following lemma.

\begin{lemma}\label{lemevenwcdim}
Let $\mathcal{C}$ be a $(v_2)$ configuration. Then, 
\[
\mathcal{C} = \bigcup_{i=1}^t C_{n_i},
\]
where $n_i>2$, for all $1\leq i \leq t$. Moreover,
\[
wcdim(Levi_{\mathcal{C}}) = \sum_{i=1}^t wcdim(C_{2n_i})
\]
\end{lemma}

Finally, we notice that Theorem 5 in \cite{BKMUV} implies
\[
wcdim(C_{2n})=
\left\{ \begin{array}{cl}
2 & $if$ \ n=3 \\  
0  & $if$ \ n\geq 4
\end{array}
\right.
\]

Next is an immediate corollary of Theorem 5 in \cite{BKMUV} and Lemma \ref{lemevenwcdim}.

\begin{corollary}
$wcdim(Levi_{\mathcal{C}})$ is even, for all $(v_2)$ configurations $\mathcal{C}$. \\
Moreover, for every $n\in \mathbb{N}$, there is a $(v_2)$ configuration, $\mathcal{C}_n$, such that 
\[
wcdim(Levi_{\mathcal{C}_n}) =  2n
\]
In particular, the sequence $\{wcdim(Levi_{\mathcal{C}_n})\}_{n=1}^{\infty}$ is unbounded.
\end{corollary}

We conclude that Theorem \ref{PH(rw)} cannot be expanded to the case $r=2$. However, it is still an open problem to find the well-covered  dimension of all Levi graphs of $(v_2,b_k)$ configurations.

Of course, the study of the well-covered dimension of Levi graphs of configurations not of the form $(v_r, b_k)$ is also an interesting open problem.

\vspace{.1in}

Block designs are another family of objects that could be studied to attempt a generalization of Theorem  \ref{PH(rw)}.

\begin{definition}
Let $\lambda, t\geq 1$. A $t-(v, k, \lambda)$ design (or $t$-design), is an incidence structure of points and blocks with the following properties:
\begin{enumerate}
\item there are $v$ points;
\item each block is incident with $k$ points;
\item any $t$ points are incident with $\lambda$ common blocks.
\end{enumerate}
\end{definition}

It is easy to see that a $1-(v, k, \lambda)$ design is a $(v_{\lambda},b_k)$ configuration, where $b=v\lambda /k$. Moreover, a $2-(v, k, 1)$ design is a configuration in which every pair of points are `collinear.' For $t>1$ and $\lambda >1$, the obvious definition of the Levi graph of a $t$-design would yield a multigraph. This apparent setback is not so much of a problem since having one edge or multiple edges between two vertices would mean the same when looking for maximal independent sets. We claim that the ideas used to prove Theorem \ref{PH(rw)} can be generalized to be applicable to block designs.

\vspace{.1in}

Finally, in this work, we study the well-covered space of the Levi graph of any given configuration. We propose, as an interesting open problem, the study of configurations via understanding the well-covered spaces of their collinearity graphs (in which points in a configuration are defined as vertices, and adjacency occurs if and only if the points are collinear). The third author is currently working on a particular case of this problem: generalized quadrangles.


\end{document}